\documentclass[12pt,twoside]{article}
\usepackage{a4}
\usepackage{amssymb,amsmath,amsthm,latexsym,tikz}
\usepackage{amsfonts}
  \usepackage{amsfonts}
\usepackage{graphicx}
\usetikzlibrary{arrows,decorations.markings}
\newtheorem{theorem}{Theorem}[section]

\newtheorem{corollary}[theorem] {Corollary}

\newtheorem{lemma} [theorem]{Lemma}

\vspace{5cm}

\begin{document}
  
  \label{'ubf'}  
\setcounter{page}{1}

\markboth {\hspace*{-9mm} \centerline{\footnotesize \sc
   Decomposition of hypercubes into sunlet graphs}
                 }
                { \centerline                           {\footnotesize \sc  
         A.V. Sonawane                                        } \hspace*{-9mm}              
               }
\begin{center}
{ 
       {\Large \textbf { \sc  Decomposition of hypercubes\\ into sunlet graphs}
       }
\\
\medskip
{\sc A.V. Sonawane}\\
{\footnotesize Government of Maharashtra's Ismail Yusuf College of Arts, Science and Commerce,\\ Mumbai 400 060, INDIA.}\\

{\footnotesize e-mail: {\it amolvson@gmail.com}}
}
\end{center}

\thispagestyle{empty}


\begin{abstract}  
{\footnotesize  For any positive integer $k \geq 3,$ the sunlet graph of order $2k$, denoted by $L_{2k},$ is the graph obtained by adding a pendant edge to each vertex of a cycle of length $k.$ In this paper, we prove that the necessary and sufficient condition for the existence of an $L_{16}$-decomposition of the $n$-dimensional hypercube $Q_n$ is $n = 4$ or $n \geq 6.$ Also, we prove that for any integer $m \geq 2,$ $Q_{mn}$ has an $L_{2k}$-decomposition if $Q_{n}$ has a $C_k$-decomposition.

\vspace{0.3cm}

}
\end{abstract}

\section{Introduction}
All graphs under consideration are simple and finite. For any positive integer $n,$ the {\it hypercube} of dimension $n,$ denoted by $Q_n,$ is a graph with vertex set $\{x_1 x_2 \cdots x_n : x_i =$ $0$ or $1$ for $i = 1, 2, \cdots, n \}$ and any two vertices are adjacent in $Q_n$ if and only if they differ at exactly one position. The {\it Cartesian product} of graphs $G$ and $H,$ denoted by $G \Box H,$ is a graph with vertex set $V(G) \times V(H),$ and two vertices $(x,y)$ and $(u,v)$ are adjacent in $G \Box H$ if and only if either $x=u$ and $y$ is adjacent to $v$ in $H,$ or $x$ is adjacent to $u$ in $G$ and $y=v.$ It is well-known that $Q_n$ is the Cartesian product of $n$ copies of the complete graph $K_2.$ Note that $Q_n$ is an $n$-regular and $n$-connected graph with $2^n$ vertices and $n 2^{n-1}$ edges.

Let $k \geq 3$ be an integer. A cycle of length $k$ is denoted by $C_k.$ The {\it sunlet graph} of order $2k,$ denoted by $L_{2k},$ is obtained by adding a pendant edge to each vertex of the cycle $C_k$ \cite{a}. Note that $L_{2k}$ has $2k$ vertices and $2k$ edges. The sunlet graph of order sixteen $L_{16}$ is shown in Figure 1.

\vspace{0.2cm}
\begin{center}
    \scalebox{0.5}{
        \begin{tikzpicture}
        \draw [fill=black] (0,0) circle (0.1);
        \draw [fill=black] (1,0) circle (0.1);
        \draw [fill=black] (2,1) circle (0.1);
        \draw [fill=black] (2,2) circle (0.1);
        \draw [fill=black] (1,3) circle (0.1);
        \draw [fill=black] (0,3) circle (0.1);
        \draw [fill=black] (-1,2) circle (0.1);
        \draw [fill=black] (-1,1) circle (0.1);
        
        \draw [fill=black] (-0.5,-1) circle (0.1);
        \draw [fill=black] (1.5,-1) circle (0.1);
        \draw [fill=black] (3,0.5) circle (0.1);
        \draw [fill=black] (3,2.5) circle (0.1);
        \draw [fill=black] (1.5,4) circle (0.1);
        \draw [fill=black] (-0.5,4) circle (0.1);
        \draw [fill=black] (-2,2.5) circle (0.1);
        \draw [fill=black] (-2,0.5) circle (0.1);
        
        \draw (0,0)--(1,0)--(2,1)--(2,2)--(1,3)--(0,3)--(-1,2)--(-1,1)--(0,0) (0,0)--(-0.5,-1) (1,0)--(1.5,-1) (2,1)--(3,0.5) (2,2)--(3,2.5) (1,3)--(1.5,4) (0,3)--(-0.5,4) (-1,2)-- (-2,2.5) (-1,1)--(-2,0.5);
        
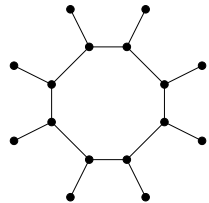
        \node at (0.5,-2) {\Large Figure 1. The sunlet graph $L_{16}$};
        \end{tikzpicture}}
\end{center}

A {\it decomposition} of a graph $G$ is a collection of edge-disjoint subgraphs of $G$ such that the edge set of the subgraphs partitions the edge set of $G.$ For a given graph $H,$ an {\it $H$-decomposition} of  $G$ is a decomposition into subgraphs each isomorphic to $H.$ 

The problem of decomposing the given graph into the sunlet graphs is studied for various classes of regular graphs in the literature \cite{a, ani, c, f, s, m}. Fu et al. \cite{f} proved that if $k = 6,10,14$ or $2^m ~(m \geq 2),$ then there exists an $L_{2k}$-decomposition of $K_n$ if and only if $n \geq 2k$ and $n(n-1) \equiv 0 (\text{mod}~ 4k).$ The existence of an $L_{10}$-decomposition of the complete graph $K_n$ for $n \equiv 0, 1, 5, 16 ({\text{mod}}~20)$ is guaranteed by Fu, Huang and Lin \cite{c}. Anitha and Lekshmi \cite{ani} established that the complete graph $K_{2n},$ the complete bipartite graph $K_{2n,2n}$ and the Harary graph $H_{4, 2n}$ have $L_{2n}$-decompositions for all $n \geq 3.$ Akwu and Ajayi \cite{a} proved that for even $m \geq 2,$ odd $n \geq 3$ and odd prime $p,$ the lexicographic product of $K_n$ and the graph $\Bar{K}_m$ consisting of only $m$ isolated vertices has an $L_{2p}$-decomposition if and only if $\frac{1}{2} n (n-1) m^2 \equiv 0 (\text{mod}~ 2p).$ Sowndhariya and Muthusamy \cite{sm} gave necessary and sufficient conditions for the existence of an $L_8$-decomposition
of tensor product and wreath product of complete graphs. Sowndhariya and Muthusamy \cite{m} studied an $L_8$-decomposition of the graph $K_n \Box K_m$ and proved that such a decomposition exists if and only if $n$ and $m$ satisfy one of the specific eight conditions. Sonawane and Borse \cite{s} proved that the $n$-dimensional hypercube $Q_n$ has an $L_8$-decomposition if and only if $n$ is 4 or $n \geq 6.$

In this paper, we consider the problem of decomposing the hypercube $Q_n$ into the sunlet graphs. In Section 2, we prove that the necessary and sufficient condition for the existence of an $L_{16}$-decomposition of $Q_n$ is $n = 4$ or $n \geq 6.$ In Section 3, we prove that if $Q_{n}$ has a $C_k$-decomposition, then $Q_{mn}$ has an $L_{2k}$-decomposition for $m \geq 2.$

\section{An $L_{16}$-decomposition of hypercubes}

In this section, we prove that the necessary and sufficient condition for the existence of an $L_{16}$-decomposition of $Q_n$ is $n = 4$ or $n \geq 6.$

We need a corollary of the following result due to El-Zanati and Eynden \cite{z}. They considered the cycle decomposition of the Cartesian product of cycles each of length power of $2$ and obtained the result, which is stated below.

\begin{theorem}
Let $n, k_1, k_2, \cdots, k_n \geq 2$ be integers and let $G$ be the Cartesian product of the cycles $C_{2^{k_1}}, C_{2^{k_2}}, \cdots C_{2^{k_n}}.$ Then there exists a $C_s$-decomposition of $G$ if and only if $s = 2^t$ with $2 \leq t \leq  k_1 + k_2 + · · · + k_n.$
\end{theorem}

The following result is a corollary of the above theorem as $Q_n$ is the Cartesian product of $\frac{n}{2}$ cycles of length $4$ for any even integer $n \geq 2.$

\begin{corollary} \label{C}
For any even integer $n \geq 2,$ there exists a $C_s$-decomposition of $Q_n$ if and only if $s = 2^t$ with $2 \leq t \leq 2^n.$
\end{corollary}

In the next lemma, we prove that the necessary condition for the existence of an $L_{16}$-decomposition of $Q_n$ is $n = 4$ or $n \geq 6.$
\begin{lemma} \label{5}
There does not exist an $L_{16}$-decomposition of $Q_n$ if $n \in \{1,2,3,5\}.$
\end{lemma}

\begin{proof} 
Contrary assume that $Q_n$ has an $L_{16}$-decomposition for some $n \in \{1,2,3,5\}.$ Then the number of edges of $L_{16}$ must divide the number of edges of $Q_n.$ Hence $16$ divides $n 2^{n-1}.$ This shows that $n \geq 4$ and so, $n=5.$ Since $Q_5$ has $80$ edges, there are five copies of the graph $L_{16}$ in the $L_{16}$-decomposition of $Q_5.$ Every vertex of $Q_5$ has degree $5$ whereas $L_{16}$ has eight vertices of degree 3 and eight of degree 1. Therefore, a degree 3 vertex of any copy of $L_{16}$ in the decomposition cannot be a degree 3 vertex of another copy of $L_{16}.$ This implies that $Q_5$ has at least 40 vertices, a contradiction.
\end{proof}

In the next lemma, we give decomposition of $C_k \Box C_k$ into spanning sunlet subgraphs for any even integer $k \geq 4.$
\begin{lemma} \label{2}
For any even integer $k \geq 4,$ the graph $C_k \Box C_k$ has an $L_{k^2}$-decomposition.
\end{lemma}

\begin{proof}
Let $V(C_k) = \mathbb{Z}_k$ such that a vertex $i$ is adjacent to a vertex $i+1\pmod{k}.$ Then $V(C_k \Box C_k) = \{(i,j) : i,j = 1,2,\cdots,k\}.$ We construct two vertex-disjoint cycles $Z_1$ and $Z_2$ of length $\frac{k^2}{2}$ in $C_k \Box C_k$ as $Z_1=\langle (1,1),(1,2),\cdots,(1,\frac{k}{2}),(2,\frac{k}{2}),(2,\frac{k}{2}+1),\cdots,(2,k-1),(3,k-1),(3,k),(3,1),\cdots,(3,\frac{k}{2}-2),\cdots,(k,1) \rangle$ and $Z_2= \langle (1,\frac{k}{2}+1),(1,\frac{k}{2}+2),\cdots,(1,k),\\(2,k),(2,1),\cdots,(2,\frac{k}{2}-1),(3,\frac{k}{2}-1),(3,\frac{k}{2}), \cdots,(3,k-1),\cdots,(k,\frac{k}{2}+1) \rangle.$ Now we adjoin a pendant edge to each vertex of $Z_1$ and $Z_2$ in the lexicographic order as per the availability of the vertex, so that we get two edge-disjoint spanning subgraphs of $C_k \Box C_k$ which are isomorphic to $L_{k^2}.$ This completes the proof.
\end{proof}

For an illustration, an $L_{64}$-decomposition of $C_8 \Box C_8$ is shown in Figure 2. For convenience, edges of the cycles $C_{32}$ are shown by lines and edges with the pendant vertices by dotted lines in both the copies of $L_{64}.$

\vspace{0.2cm}
\begin{center}
    \scalebox{0.7}{
        \begin{tikzpicture}
        \draw [fill=black] (0,0) circle (0.1);
        \draw [fill=black] (1,0) circle (0.1);
        \draw [fill=black] (2,0) circle (0.1);
        \draw [fill=black] (3,0) circle (0.1);
        \draw [fill=black] (4,0) circle (0.1);
        \draw [fill=black] (5,0) circle (0.1);
        \draw [fill=black] (6,0) circle (0.1);
        \draw [fill=black] (7,0) circle (0.1);
        
        \draw [fill=black] (0,1) circle (0.1);
        \draw [fill=black] (1,1) circle (0.1);
        \draw [fill=black] (2,1) circle (0.1);
        \draw [fill=black] (3,1) circle (0.1);
        \draw [fill=black] (4,1) circle (0.1);
        \draw [fill=black] (5,1) circle (0.1);
        \draw [fill=black] (6,1) circle (0.1);
        \draw [fill=black] (7,1) circle (0.1);
        
        \draw [fill=black] (0,2) circle (0.1);
        \draw [fill=black] (1,2) circle (0.1);
        \draw [fill=black] (2,2) circle (0.1);
        \draw [fill=black] (3,2) circle (0.1);
        \draw [fill=black] (4,2) circle (0.1);
        \draw [fill=black] (5,2) circle (0.1);
        \draw [fill=black] (6,2) circle (0.1);
        \draw [fill=black] (7,2) circle (0.1);
        
        \draw [fill=black] (0,3) circle (0.1);
        \draw [fill=black] (1,3) circle (0.1);
        \draw [fill=black] (2,3) circle (0.1);
        \draw [fill=black] (3,3) circle (0.1);
        \draw [fill=black] (4,3) circle (0.1);
        \draw [fill=black] (5,3) circle (0.1);
        \draw [fill=black] (6,3) circle (0.1);
        \draw [fill=black] (7,3) circle (0.1);
        
        \draw [fill=black] (0,4) circle (0.1);
        \draw [fill=black] (1,4) circle (0.1);
        \draw [fill=black] (2,4) circle (0.1);
        \draw [fill=black] (3,4) circle (0.1);
        \draw [fill=black] (4,4) circle (0.1);
        \draw [fill=black] (5,4) circle (0.1);
        \draw [fill=black] (6,4) circle (0.1);
        \draw [fill=black] (7,4) circle (0.1);
        
        \draw [fill=black] (0,5) circle (0.1);
        \draw [fill=black] (1,5) circle (0.1);
        \draw [fill=black] (2,5) circle (0.1);
        \draw [fill=black] (3,5) circle (0.1);
        \draw [fill=black] (4,5) circle (0.1);
        \draw [fill=black] (5,5) circle (0.1);
        \draw [fill=black] (6,5) circle (0.1);
        \draw [fill=black] (7,5) circle (0.1);
        
        \draw [fill=black] (0,6) circle (0.1);
        \draw [fill=black] (1,6) circle (0.1);
        \draw [fill=black] (2,6) circle (0.1);
        \draw [fill=black] (3,6) circle (0.1);
        \draw [fill=black] (4,6) circle (0.1);
        \draw [fill=black] (5,6) circle (0.1);
        \draw [fill=black] (6,6) circle (0.1);
        \draw [fill=black] (7,6) circle (0.1);
        
        \draw [fill=black] (0,7) circle (0.1);
        \draw [fill=black] (1,7) circle (0.1);
        \draw [fill=black] (2,7) circle (0.1);
        \draw [fill=black] (3,7) circle (0.1);
        \draw [fill=black] (4,7) circle (0.1);
        \draw [fill=black] (5,7) circle (0.1);
        \draw [fill=black] (6,7) circle (0.1);
        \draw [fill=black] (7,7) circle (0.1);
        
        \draw[line width=0.3mm] (0,0)--(0,3)--(1,3)--(1,6)--(2,6)--(2,7)..controls(2.8,3.5)..(2,0)--(2,1)--(3,1)--(3,4)--(4,4)--(4,7)--(5,7)..controls(5.8,3.5)..(5,0)--(5,2)--(6,2)--(6,5)--(7,5)--(7,7)..controls(7.8,3.5)..(7,0)..controls(3.5,-0.8)..(0,0);
        
        \draw[dotted] [line width=0.4mm] (0,0) -- (1,0) (0,1) -- (1,1) (0,2) -- (1,2) (0,3) -- (0,4) (1,3) -- (2,3) (1,4) -- (2,4) (1,5) -- (2,5) (1,6) -- (1,7) (2,0) -- (3,0) (2,1) -- (2,2) (2,6) -- (3,6) (2,7) -- (3,7) (3,1)--(4,1) (3,2)--(4,2) (3,3)--(4,3) (3,4)--(3,5) (4,4)--(5,4) (4,5)--(5,5) (4,6)--(5,6) (4,7)..controls(4.8,3.5)..(4,0) (5,0)--(6,0) (5,1)--(6,1) (5,2)--(5,3) (5,7)--(6,7) (6,2)--(7,2) (6,3)--(7,3) (6,4)--(7,4) (6,5)--(6,6) (7,0)--(7,1) (7,5)..controls(3.5,5.8)..(0,5) (7,6)..controls(3.5,6.8)..(0,6) (7,7)..controls(3.5,7.8)..(0,7);
        
        \draw [fill=black] (10,0) circle (0.1);
        \draw [fill=black] (11,0) circle (0.1);
        \draw [fill=black] (12,0) circle (0.1);
        \draw [fill=black] (13,0) circle (0.1);
        \draw [fill=black] (14,0) circle (0.1);
        \draw [fill=black] (15,0) circle (0.1);
        \draw [fill=black] (16,0) circle (0.1);
        \draw [fill=black] (17,0) circle (0.1);
        
        \draw [fill=black] (10,1) circle (0.1);
        \draw [fill=black] (11,1) circle (0.1);
        \draw [fill=black] (12,1) circle (0.1);
        \draw [fill=black] (13,1) circle (0.1);
        \draw [fill=black] (14,1) circle (0.1);
        \draw [fill=black] (15,1) circle (0.1);
        \draw [fill=black] (16,1) circle (0.1);
        \draw [fill=black] (17,1) circle (0.1);
        
        \draw [fill=black] (10,2) circle (0.1);
        \draw [fill=black] (11,2) circle (0.1);
        \draw [fill=black] (12,2) circle (0.1);
        \draw [fill=black] (13,2) circle (0.1);
        \draw [fill=black] (14,2) circle (0.1);
        \draw [fill=black] (15,2) circle (0.1);
        \draw [fill=black] (16,2) circle (0.1);
        \draw [fill=black] (17,2) circle (0.1);
        
        \draw [fill=black] (10,3) circle (0.1);
        \draw [fill=black] (11,3) circle (0.1);
        \draw [fill=black] (12,3) circle (0.1);
        \draw [fill=black] (13,3) circle (0.1);
        \draw [fill=black] (14,3) circle (0.1);
        \draw [fill=black] (15,3) circle (0.1);
        \draw [fill=black] (16,3) circle (0.1);
        \draw [fill=black] (17,3) circle (0.1);
        
        \draw [fill=black] (10,4) circle (0.1);
        \draw [fill=black] (11,4) circle (0.1);
        \draw [fill=black] (12,4) circle (0.1);
        \draw [fill=black] (13,4) circle (0.1);
        \draw [fill=black] (14,4) circle (0.1);
        \draw [fill=black] (15,4) circle (0.1);
        \draw [fill=black] (16,4) circle (0.1);
        \draw [fill=black] (17,4) circle (0.1);
        
        \draw [fill=black] (10,5) circle (0.1);
        \draw [fill=black] (11,5) circle (0.1);
        \draw [fill=black] (12,5) circle (0.1);
        \draw [fill=black] (13,5) circle (0.1);
        \draw [fill=black] (14,5) circle (0.1);
        \draw [fill=black] (15,5) circle (0.1);
        \draw [fill=black] (16,5) circle (0.1);
        \draw [fill=black] (17,5) circle (0.1);
        
        \draw [fill=black] (10,6) circle (0.1);
        \draw [fill=black] (11,6) circle (0.1);
        \draw [fill=black] (12,6) circle (0.1);
        \draw [fill=black] (13,6) circle (0.1);
        \draw [fill=black] (14,6) circle (0.1);
        \draw [fill=black] (15,6) circle (0.1);
        \draw [fill=black] (16,6) circle (0.1);
        \draw [fill=black] (17,6) circle (0.1);
        
        \draw [fill=black] (10,7) circle (0.1);
        \draw [fill=black] (11,7) circle (0.1);
        \draw [fill=black] (12,7) circle (0.1);
        \draw [fill=black] (13,7) circle (0.1);
        \draw [fill=black] (14,7) circle (0.1);
        \draw [fill=black] (15,7) circle (0.1);
        \draw [fill=black] (16,7) circle (0.1);
        \draw [fill=black] (17,7) circle (0.1);
        
        \draw[line width=0.3mm] (10,4)--(10,7)--(11,7)..controls(11.8,3.5)..(11,0)--(11,2)--(12,2)--(12,5)--(13,5)--(13,7)..controls(13.8,3.5)..(13,0)--(14,0)--(14,3)--(15,3)--(15,6)--(16,6)--(16,7)..controls(16.8,3.5)..(16,0)--(16,1)--(17,1)--(17,4)..controls(13.5,3.2)..(10,4);
        
        \draw[dotted] [line width=0.4mm] (10,4)--(11,4) (10,5)--(11,5) (10,6)--(11,6) (10,7)..controls(9.2,3.5)..(10,0) (11,0)--(12,0) (11,1)--(12,1) (11,2)--(11,3) (11,7)--(12,7) (12,2)--(13,2) (12,3)--(13,3) (12,4)--(13,4) (12,5)--(12,6) (13,0)--(13,1) (13,5)--(14,5) (13,6)--(14,6) (13,7)--(14,7) (14,0)--(15,0) (14,1)--(15,1) (14,2)--(15,2) (14,3)--(14,4) (15,3)--(16,3) (15,4)--(16,4) (15,5)--(16,5) (15,6)--(15,7) (16,0)--(17,0) (16,1)--(16,2) (16,6)--(17,6) (16,7)--(17,7) (17,4)--(17,5) (17,1)..controls(13.5,0.2)..(10,1) (17,2)..controls(13.5,1.2)..(10,2) (17,3)..controls(13.5,2.2)..(10,3);
       
         \node at (8.5,-1.5) {\large Figure 2. An $L_{64}$-decomposition of $C_8 \Box C_8$};
        \end{tikzpicture}}
\end{center}

The following result is a corollary of the above lemma.
\begin{corollary} \label{4}
For any integer $n \geq 1,$ there exists an $L_{2^{4n}}$-decomposition of $Q_{4n}.$ In other words, $Q_{4n}$ has a decomposition into the spanning sunlet graphs for any integer $n \geq 1.$
\end{corollary}

\begin{proof}
We can write $Q_{4n}=Q_{2n} \Box Q_{2n}.$ By Corollary \ref{C}, $Q_{2n}$ has a decomposition into Hamiltonian cycles. Let $Z_1,Z_2, \cdots, Z_n$ be Hamiltonian cycles in $Q_{2n}$ such that the collection $\{Z_1,Z_2, \cdots, Z_n\}$ decomposes $Q_{2n}.$ Then $Z_1 \Box Z_1,Z_2 \Box Z_2, \cdots, Z_n \Box Z_n$ are edge-disjoint spanning subgraphs of $Q_{4n}$ and their collection decomposes $Q_{4n}.$ By Lemma \ref{2}, each $Z_i \Box Z_i$ has an $L_{2^{4n}}$-decomposition. Hence $Q_{4n}$ has an $L_{2^{4n}}$-decomposition.
\end{proof}

Now we prove the necessary condition for the existence of an $L_{16}$-decomposition of $Q_n$ is also sufficient.

We need the following four lemmas to prove the sufficient condition.

\begin{lemma} \label{6}
There exists an $L_{16}$-decomposition of $Q_6.$
\end{lemma}
\begin{proof}
Write $Q_6$ as $Q_6 = Q_4 \Box C_4$ as $C_4 = Q_2.$ Thus $Q_6$ is obtained by replacing each vertex of $C_4$ by a copy of $Q_4$ and replacing each edge of $C_4$ by a matching between two copies of $Q_4$ corresponding to the end vertices of that edge. Let $C_4 = \langle 0,1,2,3,0 \rangle$ and $Q_4^0, Q_4^1, Q_4^2, Q_4^3$ be copies of $Q_4$ in $Q_6$ corresponding to vertices $0,1,2,3$ of $C_4,$ respectively. For $i \in \{0,2\},$ $Q_4^i$ has an $L_{16}$-decomposition by Lemma \ref{2} as each $Q_4^i$ can be written as the Cartesian product of cycles of length $4.$ For $i \in \{1,3\},$ from each vertex of $Q_4^i,$ exactly two cycles of length eight passes as $Q_4^i$ has a $C_8$-decomposition by Corollary \ref{C}. Adjoin each vertex of one of two cycles to the corresponding vertex in $Q_4^0,$ and adjoin each vertex of the other cycle to the corresponding vertex in $Q_4^2.$ So, from each copy of the cycle of length eight, we get a copy of $L_{16}.$ This completes the proof.
\end{proof}

\begin{lemma} \label{7}
There exists an $L_{16}$-decomposition of $Q_7.$
\end{lemma}
\begin{proof}
Write $Q_7$ as $Q_7 = Q_4 \Box Q_3.$ Let $D$ be a directed graph obtained from $Q_3$ by giving directions to the edges, as shown in Figure 3.

\vspace{0.2cm}
\begin{center}
    \scalebox{1}{
        \begin{tikzpicture}
        \draw [fill=black] (0,0) circle (0.1);
        \draw [fill=black] (3,0) circle (0.1);
        \draw [fill=black] (0,3) circle (0.1);
        \draw [fill=black] (3,3) circle (0.1);
        \draw [fill=black] (1,1) circle (0.1);
        \draw [fill=black] (2,1) circle (0.1);
        \draw [fill=black] (2,2) circle (0.1);
        \draw [fill=black] (1,2) circle (0.1);
       
        \draw (0,0) -- (3,0) -- (3,3) -- (0,3) -- (0,0) (1,1) -- (2,1) -- (2,2) -- (1,2) -- (1,1) -- (0,0) (2,1) -- (3,0) (2,2) -- (3,3) (1,2) -- (0,3);
        
        \draw[decoration={markings,mark=at position 1 with
        {\arrow[scale=3,>=stealth]{>}}},postaction={decorate}]
        (0.6,0) -- (0.5,0);
        \draw[decoration={markings,mark=at position 1 with
        {\arrow[scale=3,>=stealth]{>}}},postaction={decorate}]
        (0.4,0.4) -- (0.3,0.3);
        \draw[decoration={markings,mark=at position 1 with
        {\arrow[scale=3,>=stealth]{>}}},postaction={decorate}]
        (0,0.6) -- (0,0.5);
        \draw[decoration={markings,mark=at position 1 with
        {\arrow[scale=3,>=stealth]{>}}},postaction={decorate}]
        (3,2.5) -- (3,2.6);
        \draw[decoration={markings,mark=at position 1 with
        {\arrow[scale=3,>=stealth]{>}}},postaction={decorate}]
        (2.6,2.6) -- (2.7,2.7);
        \draw[decoration={markings,mark=at position 1 with
        {\arrow[scale=3,>=stealth]{>}}},postaction={decorate}]
        (2.5,3) -- (2.6,3);
        \draw[decoration={markings,mark=at position 1 with
        {\arrow[scale=3,>=stealth]{>}}},postaction={decorate}]
        (1.4,1) -- (1.3,1);
        \draw[decoration={markings,mark=at position 1 with
        {\arrow[scale=3,>=stealth]{>}}},postaction={decorate}]
        (1,1.6) -- (1,1.7);
        \draw[decoration={markings,mark=at position 1 with
        {\arrow[scale=3,>=stealth]{>}}},postaction={decorate}]
        (1.6,2) -- (1.7,2);
        \draw[decoration={markings,mark=at position 1 with
        {\arrow[scale=3,>=stealth]{>}}},postaction={decorate}]
        (2,1.4) -- (2,1.3);
        \draw[decoration={markings,mark=at position 1 with
        {\arrow[scale=3,>=stealth]{>}}},postaction={decorate}]
        (0.4,2.6) -- (0.3,2.7);
        \draw[decoration={markings,mark=at position 1 with
        {\arrow[scale=3,>=stealth]{>}}},postaction={decorate}]
        (2.6,0.4) -- (2.7,0.3);
        
        \node at (1.5,-0.6) {Figure 3.};
        
       \end{tikzpicture}}
\end{center}

In $D,$ there are two vertices with in-degree 3 and out-degree 0, and the in-degrees and out-degrees of remaining all vertices are 1 and 2, respectively. The graph $Q_7$ is obtained by replacing each vertex of $Q_3$ with a copy of $Q_4$ and replacing each edge of $Q_3$ by a matching between two copies of $Q_4$ corresponding to the end vertices of that edge. Consider an $L_{16}$-decomposition of copies of $Q_4$ corresponding to each vertex of $D$ with out-degree 0, and a $C_8$-decomposition of copies of $Q_4$ corresponding to each vertex of $D$ with out-degree 2. In a $C_8$-decomposition of copies of $Q_4,$ exactly two cycles pass from each vertex. Adjoin a pedant edge to each vertex of copies of $Q_4$ of a vertex corresponding the out-degree 2, to one of the vertices of its nearest copy of $Q_4$ according to the direction of the corresponding edge in $D.$ Then we get $L_{16}$ from each $C_8$ from a $C_8$-decomposition of each copy of $Q_4$ of a vertex corresponding to the out-degree 2. Hence we get an $L_{16}$-decomposition of $Q_7.$
\end{proof}

\begin{lemma} \label{9}
There exists an $L_{16}$-decomposition of $Q_9.$
\end{lemma}
\begin{proof}
Write $Q_9$ as $Q_9 = Q_6 \Box Q_3.$ Let $D$ be a directed graph obtained from $Q_3$ by giving directions to the edges, as shown in Figure 4.

\vspace{0.2cm}
\begin{center}
    \scalebox{1}{
        \begin{tikzpicture}
        \draw [fill=black] (0,0) circle (0.1);
        \draw [fill=black] (3,0) circle (0.1);
        \draw [fill=black] (0,3) circle (0.1);
        \draw [fill=black] (3,3) circle (0.1);
        \draw [fill=black] (1,1) circle (0.1);
        \draw [fill=black] (2,1) circle (0.1);
        \draw [fill=black] (2,2) circle (0.1);
        \draw [fill=black] (1,2) circle (0.1);
       
        \draw (0,0) -- (3,0) -- (3,3) -- (0,3) -- (0,0) (1,1) -- (2,1) -- (2,2) -- (1,2) -- (1,1) -- (0,0) (2,1) -- (3,0) (2,2) -- (3,3) (1,2) -- (0,3);
        
        \draw[decoration={markings,mark=at position 1 with
        {\arrow[scale=3,>=stealth]{>}}},postaction={decorate}]
        (0.6,0) -- (0.5,0);
        \draw[decoration={markings,mark=at position 1 with
        {\arrow[scale=3,>=stealth]{>}}},postaction={decorate}]
        (0.4,0.4) -- (0.3,0.3);
        \draw[decoration={markings,mark=at position 1 with
        {\arrow[scale=3,>=stealth]{>}}},postaction={decorate}]
        (0,0.6) -- (0,0.5);
        \draw[decoration={markings,mark=at position 1 with
        {\arrow[scale=3,>=stealth]{>}}},postaction={decorate}]
        (3,2.5) -- (3,2.6);
        \draw[decoration={markings,mark=at position 1 with
        {\arrow[scale=3,>=stealth]{>}}},postaction={decorate}]
        (2.6,2.6) -- (2.7,2.7);
        \draw[decoration={markings,mark=at position 1 with
        {\arrow[scale=3,>=stealth]{>}}},postaction={decorate}]
        (2.5,3) -- (2.6,3);
        \draw[decoration={markings,mark=at position 1 with
        {\arrow[scale=3,>=stealth]{>}}},postaction={decorate}]
        (1.6,1) -- (1.7,1);
        \draw[decoration={markings,mark=at position 1 with
        {\arrow[scale=3,>=stealth]{>}}},postaction={decorate}]
        (1,1.6) -- (1,1.7);
        \draw[decoration={markings,mark=at position 1 with
        {\arrow[scale=3,>=stealth]{>}}},postaction={decorate}]
        (1.4,2) -- (1.3,2);
        \draw[decoration={markings,mark=at position 1 with
        {\arrow[scale=3,>=stealth]{>}}},postaction={decorate}]
        (2,1.4) -- (2,1.3);
        \draw[decoration={markings,mark=at position 1 with
        {\arrow[scale=3,>=stealth]{>}}},postaction={decorate}]
        (0.7,2.3) -- (0.8,2.2);
        \draw[decoration={markings,mark=at position 1 with
        {\arrow[scale=3,>=stealth]{>}}},postaction={decorate}]
        (2.3,0.7) -- (2.2,0.8);
        
        \node at (1.5,-0.6) {Figure 4.};
       \end{tikzpicture}}
\end{center}

In $D,$ there are four vertices with out-degree 0, and the out-degree of the remaining four vertices is 3. The graph $Q_9$ is obtained by replacing each vertex of $Q_3$ with a copy of $Q_6$ and replacing each edge of $Q_3$ by a matching between two copies of $Q_6$ corresponding to the end vertices of that edge. Consider an $L_{16}$-decomposition of copies of $Q_6$ of vertices corresponding to the out-degree 0 and a $C_8$-decomposition of copies of $Q_6$ of vertices corresponding to the out-degree 3. In a $C_8$-decomposition of copies of $Q_6,$ exactly three cycles pass from each vertex. Adjoin a pedant edge to each vertex of copies of $Q_6$ corresponding to each vertex with out-degree 3, to one of the vertices of its nearest copy of $Q_6$ according to the direction of the corresponding edge in $D.$ Then we get a copy of $L_{16}$ from each copy of $C_8$ from a $C_8$-decomposition of each copy of $Q_6$ corresponding to each vertex with out-degree 3. Hence we get an $L_{16}$-decomposition of $Q_9.$
\end{proof}

The following lemma follows from the definition of the Cartesian product of graphs.
\begin{lemma} \label{def}
If  the graphs $G_1$ and  $G_2$ each has an $H$-decomposition, then the graph $G_1 \Box G_2$ has an $H$-decomposition.
\end{lemma}

In the following lemma, we prove that the sufficient condition for the existence of an $L_{16}$-decomposition of $Q_n$ is $n=4$ or $n \geq 6.$
\begin{lemma} \label{n}
There exists an $L_{16}$-decomposition of $Q_n$ if $n=4$ or $n \geq 6.$
\end{lemma}

\begin{proof}
We prove the result by induction on $n.$ For $n=4,$ the result holds  as $Q_4$ has an $L_{16}$-decomposition by Lemma \ref{2}. For $n=8,$ we write $Q_8 = Q_4 \Box Q_4$ and the result holds by Lemma \ref{def}. For $n \in \{6,7,9\},$ the result follows by Lemmas \ref{6}, \ref{7} and \ref{9}. Suppose that $n \geq 10.$ Assume that the result holds for the $k$-dimensional hypercube for any integer $k$ with $6 \leq k \leq n-1.$ Write $Q_n = Q_{n-4} \Box Q_4.$ By induction hypothesis, $Q_{n-4}$ has an $L_{16}$-decomposition as $n-4 \geq 6.$ Hence $Q_n$ has an $L_{16}$-decomposition by Lemma \ref{def}. This completes the proof.
\end{proof}

The following result follows from Lemmas \ref{5} and \ref{n}.
\begin{theorem}
The necessary and sufficient condition for the existence of an $L_{16}$-decomposition of $Q_n$ is $n = 4$ or $n \geq 6.$
\end{theorem}

\section{An $L_{2k}$-decomposition of hypercubes}
In this section, we prove that $Q_{mn}$ has an $L_{2k}$-decomposition if $Q_{n}$ has a $C_k$-decomposition for $m \geq 2.$ In next two lemmas, we prove the result for $m=2$ and $m=3.$ Note that a $C_k$-decomposition of $Q_n$ is possible only for an even integer $n \geq 2.$ For $n=2,$ $Q_n = C_4.$
\begin{lemma} \label{2n}
If $Q_{n}$ has a $C_k$-decomposition, then $Q_{2n}$ has an $L_{2k}$-decomposition. 
\end{lemma}

\begin{proof}
Suppose $Q_{n}$ has a $C_k$-decomposition. Note that in the $C_k$-decomposition of $Q_{n},$ from each vertex of $Q_{n}$ exactly $\frac{
n}{2}$ cycles passes. We can write $Q_{2n} = Q_{n} \Box Q_{n}.$ Let $W_0, W_1, \cdots, W_{2^{n}-1}$ be copies of $Q_{n}$ in $Q_{2n}$ replaced by vertices of $Q_{n}.$ Then each $W_i$ has a $C_k$-decomposition. Also, there are $n$ copies of $W_j$'s that are adjacent to $W_i$ for each $i.$

Since $Q_{n}$ is a regular and connected graph with even degree $n,$ there is a directed Eulerian circuit in $Q_{n}$ in which each of in-degree and out-degree of each vertex is $\frac{n}{2}.$ In a $C_k$-decomposition of each $W_i,$ adjoin each vertex of each cycle to exactly one vertex of the nearest copy $W_j$ of $W_i$ in $Q_{2n},$ if there is a directed edge in the directed Eulerian circuit from the vertex $i$ to the vertex $j.$ From a $C_k$-decomposition of each $W_i$'s, we get edge-disjoint copies of $L_{2k}.$ This completes the proof.
\end{proof}

We need concepts of even and odd parity vertex in the proof of the following lemma. A vertex $v = x_1 x_2 \cdots x_n$ of $Q_n$ is said to be a vertex with \textit{even (odd) parity} if there are even (odd) number of $x_i$'s are $1$ in $v.$ Let $X$ and $Y$ be subsets of vertex set of $Q_n$ containing vertices with even parity and odd parity, respectively and $X \cup Y = V(Q_n).$ Then $(X,Y)$ is a bipartition of the bipartite graph $Q_n.$

\begin{lemma} \label{3n}
If $Q_{n}$ has a $C_k$-decomposition, then $Q_{3n}$ has an $L_{2k}$-decomposition.
\end{lemma}

\begin{proof}
We can write, $Q_{3n} = Q_{2n} \Box Q_{n}.$ Let $W_0, W_1, \cdots, W_{2^{n}-1}$ be copies of $Q_{2n}$ in $Q_{3n}$ replaced by vertices of $Q_{n}.$ Let $D$ be a digraph obtained from $Q_{n}$ such that out-degree of each vertex with even parity is $n$ and odd parity is $0.$ By Lemma \ref{2n}, each $W_j$ corresponding to vertex of $Q_{n}$ with odd parity, has an $L_{2k}$-decomposition. Consider a $C_k$-decomposition of $W_j$ corresponding to vertex of $Q_{n}$ with even parity. Note that in the $C_k$-decomposition of $W_j,$ from each vertex exactly $n$ edge-disjoint cycles passes. By adjoining exactly one vertex to each cycle in $W_j$ corresponding to vertex of $Q_{n}$ with even parity, we get copies of $L_{2k}$ corresponding to each $C_k$ in the $C_k$-decomposition of $W_j.$ This completes the proof.
\end{proof}

Now, we have the following result.
\begin{theorem} \label{m}
If $Q_{n}$ has a $C_k$-decomposition, then $Q_{mn}$ has an $L_{2k}$-decomposition for $m \geq 2.$
\end{theorem}

\begin{proof}
If $m$ is multiple of $2,$ the result holds by Lemmas \ref{def} and \ref{2n} as $Q_{mn}$ is the Cartesian product of $\frac{m}{2}$ copies of $Q_{2n}.$ Similarly, the result holds by Lemmas \ref{def} and \ref{3n} if $m$ is multiple of $3$ as $Q_{mn}$ is the Cartesian product of $\frac{m}{3}$ copies of $Q_{2n}.$ For $m=5$ and $7,$ we can write $Q_{mn}$ as  $Q_{5n}= Q_{2n} \Box Q_{3n}$ and $Q_{7n}= Q_{4n} \Box Q_{3n},$ respectively. Thus the result holds by Lemmas \ref{def}, \ref{2n} and \ref{3n} for $m=5,7.$ It follows that the result holds for $m$ with $2 \leq m \leq 10.$ Suppose that $m \geq 11,$ and $m$ is not multiple of $2$ and $3.$ Then either $m = 6q+5$ for some $q \geq 1$ or $m=6q+1$ for some $q \geq 2.$ Suppose $m = 6q+5$ for $q \geq 1.$ Then we can write $Q_{mn}$ as  $Q_{mn}= Q_{6qn} \Box Q_{5n}.$ Suppose $m = 6q+1$ for $q \geq 2.$ Then we can write $Q_{mn}$ as  $Q_{mn}= Q_{6(q-1)n} \Box Q_{7n}.$ Note that for any $r \geq 1,$ $Q_{6rn}$ has an $L_{2k}$-decomposition by both Lemmas \ref{2n} and \ref{3n}. Thus by Lemma \ref{def}, $Q_{mn}$ has an $L_{2k}$-decomposition.
\end{proof}

As a consequence of Theorem \ref{m}, we have the following result.
\begin{corollary}
Let $m \geq 2$ be an integer and $n \geq 4$ be an even integer.
\begin{enumerate}
\item $Q_{mn}$ has an $L_{2^{t+1}}$-decomposition for $2\leq t \leq n-1.$
\item $Q_{mn}$ has an $L_{2n}$-decomposition.
\item $Q_{mn}$ has an $L_{4n}$-decomposition.
\item $Q_{mn}$ has an $L_{8n}$-decomposition.
\item $Q_{mn}$ has an $L_{n 2^{k+1}}$-decomposition for $2n \leq n 2^k \leq \frac{2^n}{n}.$
\end{enumerate}
\end{corollary}
\begin{proof} 
We have following $C_k$-decompositions of $Q_n$ for an even integer $n \geq 4.$

\begin{enumerate}
    \item  Zanati and Eynden \cite{z} proved that $Q_n$ has a $C_{2^t}$-decomposition for $2\leq t \leq n-1.$
    \item Ramras \cite{r} proved that $Q_n$ has a $C_n$-decomposition.
    \item Mollard and Ramras \cite{mr} proved that $Q_n$ has a $C_{2n}$-decomposition.
    \item Tapadia, Borse and Waphare \cite{t} obtained that $Q_n$ has a $C_{4n}$-decomposition.
    \item Axenovich, Offner and Tompkins \cite{ax} established that $Q_n$ has a $C_{n 2^k}$-decomposition for $2n \leq n 2^k \leq \frac{2^n}{n}.$
\end{enumerate}

By applying Theorem \ref{m} to each of above $C_k$-decompositions of $Q_n,$ we get the desired $L_{2k}$-decomposition of $Q_{mn}.$
\end{proof}

\end{document}